\documentclass[a4paper,12pt,reqno]{amsart}
\usepackage{amsmath,amsfonts,amssymb,amsthm,enumerate,multicol}
\usepackage{tikz,graphicx}
\usepackage{stackengine}
\usepackage{subfigure,hyperref}
\usepackage{caption,enumitem,wasysym}
\usepackage{amsmath,amssymb}
\usepackage{cleveref}
\usepackage{mathrsfs}
\usepackage{physics}
\usepackage{graphics,graphicx}
\usepackage{stackengine}
\usepackage{bigints}
\usepackage{geometry}
\usepackage{epstopdf}
\usepackage{epsfig}
\usepackage{mathtools}
\usepackage{ragged2e}
\usepackage[super]{nth}
\usepackage{float,wrapfig}
\usepackage[labelsep=space]{caption}
\usepackage[font=footnotesize]{caption}
\usepackage{xcolor}
\newtheorem{lemma}{Lemma}
\newtheorem{theorem}{Theorem}
\newtheorem{definition}{Definition}
\newtheorem{corollary}{Corollary}
\newtheorem{remark}{Remark}
\newtheorem{problem}{Problem}

\title{Inequalities involving a Ramanujan Integral}
\author{Deepshikha Mishra$^\dagger$}
\address{$^\dagger$Department of Mathematics\\ Indian Institute of Technology, Roorkee-247667, Uttarakhand, India}
\email{deepshikha\_m@ma.iitr.ac.in, deepshikhamishraalld@gmail.com}

\author[A. Swaminathan]{A. Swaminathan $^{\#\ddagger}$}{\thanks{$^{\#}$Corresponding author}}
\address{$^\ddagger$Department of Mathematics\\ Indian Institute of Technology, Roorkee-247667, Uttarakhand, India}
\email{a.swaminathan@ma.iitr.ac.in, mathswami@gmail.com}
\allowdisplaybreaks
\bigskip

\begin{document}
	
	
\begin{abstract}
In this manuscript, various properties of the Ramanujan integral $I_R(x)$, defined as
\begin{align*}
		I_R(x) = \int_0^\infty e^{-xt} \dfrac{dt}{t(\pi^2 + \log^2 t)}, \quad x>0,
\end{align*}
are investigated, including its monotonicity, subadditivity, as well as convexity. Furthermore, it is shown that the Ramanujan integral admits an antiderivative that belongs to the class of Bernstein functions. Subsequently, we examine a Turan-type function involving the Ramanujan integral given by
\begin{align*}
		H_n(x;\alpha) = \left(I_R^{(n)}(x)\right)^2 - \alpha I_R^{(n-1)}(x) I_R^{(n+1)}(x), \quad x>0,
\end{align*}
	and establish its complete monotonicity under certain conditions on $\alpha$. Graphical evidences are given for the results where few ranges are yet to be established, providing scope for future research.
\end{abstract}
\maketitle
\markboth{Deepshikha Mishra and A. Swaminathan}{Inequalities involving a Ramanujan Integral}

Among numerous integrals studied by Ramanujan \cite{hardy_rama int_2015}, we consider the definite integral
\begin{align}\label{ramanujan int eq}
	I_R(x)=\int_0^\infty e^{-xt} \dfrac{dt}{t(\pi^2+\log^2 t)}, \quad x>0.
\end{align}
Other expressions for $I_R(x)$ are also available and given in \cite{ryzhik_table of integrals book, wood_1966_ramanujan int_moc}. In this work, we focus on the representation \eqref{ramanujan int eq}. This integral has various applications \cite{Giles_1995_rama int application,
Smith_2000_rama int asymp} like the one application appears in the solution of the two-dimensional diffusion equation, which relies on the following asymptotic expansion of the integral.



\begin{theorem}\cite{Bouwkamp_1971_rama int asymp}
Let $n \geq 0$ and $a > 0$. Then, as $x \to \infty$, the following holds
\begin{align}\label{rama asym}
		\int_{0}^{\infty} e^{-xt} \dfrac{t^{n-1}}{a^2 + \log^2 t}  dt \sim x^{-n} \sum_{k=0}^{\infty} \Phi_k(a, n) \log^{-k-1} x,
\end{align}
	where the coefficients $\Phi_k(a, n)$ are determined by the generating function
\begin{align}\label{rama asym coeff}
		\sum_{k=0}^{\infty} \Phi_k(a, n) \dfrac{x^k}{k!} = \dfrac{\sin(ax)}{a} \Gamma(n + x).
\end{align}
\end{theorem}
\begin{remark}
	In particular, choosing $a = \pi$ provides the asymptotic expansion of the Ramanujan integral and its derivatives. From \eqref{rama asym coeff}, we know the values of $\Phi_0(\pi,n)$, $\Phi_1(\pi,n)$ and $\Phi_2(\pi,n)$, which are given as
\begin{align}\label{value of phi}
		\Phi_0(\pi,n) = 0, \quad \Phi_1(\pi,n) = \Gamma(n), \quad \Phi_2(\pi,n) = \Gamma(n) \psi(n),
\end{align}
where $\psi(n)$ denotes the digamma function, which is the logarithmic derivative of the gamma function.
\end{remark}


Findings of Srinivasa Ramanujan are being analyzed in the literature with various objectives. For example, in \cite{cohl_rama int_2025}, a specific Ramanujan integral involving entire functions is considered in terms of bilateral hypergeometric functions and are evaluated using beta integrals for specific values. In \cite{mehrez_ramanujan type entire fun}, another class of Ramanujan type entire functions were considered for studying various aspects, including the admissibility of Turan type inequalities. Study of Ramanujan integrals is another aspect of research that exist in the literatures. For example, the derivatives of the Ramanujan integral \eqref{ramanujan int eq}, and the results concerning their complete monotonicity are established in \cite{gautschi_2024_ramanujan int_moc}. Motivated by this, we further investigate the subclass of completely monotonic functions and derive some new results on the Ramanujan integral. We first present the fundamental definitions and theorems that will be used in the proofs of the main results.

\begin{definition}\cite{Widder_1941_The Laplace Transform}\label{cmf definition}
A function $f$ is said to be completely monotonic on $(a,b)$, if it is infinitely differentiable i.e. $f \in C^\infty$ and satisfies the condition
\begin{align*}
		(-1)^k f^{(k)}(x)\geq 0,
\end{align*}
for every non-negative integer $k$.
\end{definition}

In addition to its definition, completely monotonic functions are also characterized by the Bernstein theorem, which is stated below.
\begin{theorem}\cite{Schilling_2012_bernstein function}\label{Bernstein theorem}
	Let $f:(0,\infty) \mapsto\mathbb{R}$ be a completely monotonic function. Then  it can be expressed as the Laplace transform of a unique measure $\mu$ defined on the interval $[0,\infty)$. In other words, for every $x > 0,$
	\begin{align*}
		f(x)=\mathscr{L}(\mu; x)= \int_{[0,\infty)} e^{-xt} d\mu(t).
	\end{align*}	
	Conversely, if $\mathscr{L}{(\mu; x)}< \infty $ for every $x>0$, then the function $x \mapsto \mathscr{L}{(\mu; x)}$ is completely monotone.
\end{theorem}

It is well known that completely monotone functions are closely related to completely monotone sequences. In particular, if $f(x)$ is completely monotone on $(a,\infty)$ for some $a>0$, then for any positive real $r$, the sequence $\{f(a+rn)\}$ is completely monotone \cite{Widder_1941_The Laplace Transform}. Moreover, such sequences, often referred to as Hausdorff moment sequences, play crucial role in the theory of moment problems. Specifically, the Hausdorff moment problem admits a solution if, and only if, the underlying sequence is a Hausdorff moment sequence \cite{akhiezer_book_moment problem, swami_baricz_jca_2014}.

It may also be noted that results related to the complete monotonicity of various other special functions exist in the literature. For example, \cite{Alzer_1998_Inequality polygamma_siam} explores completely monotonic behaviour of polygama functions. Based on the results given in \cite{2002_AnnAcadFenn_AlzerBerg_CM-seq}, further exploration was made in \cite{Alzer_on cm fun_rama jour_2006}, for analyzing the complete monotonicity behaviour of specific classes of functions that can be defined in terms of gamma, digamma and polygamma functions. For various other results in this direction, we refer to \cite{Koumandos_survey on cm, pedresan_2009_scm_jmaa}.

In \cite{gautschi_2024_ramanujan int_moc}, a result establishing the complete monotonicity of the Ramanujan integral is presented, which is stated as follows.
\begin{theorem}\label{ramanujan int cm thm}\cite{gautschi_2024_ramanujan int_moc}
The Ramanujan integral $I_R(x)$, as defined in \eqref{ramanujan int eq}, is completely monotonic on the interval $(0, \infty)$, and satisfies
\begin{align*}
		(-1)^m I_R^{(m)}(x) > 0 \quad \text{for all} \  m = 0, 1, 2, \ldots,
\end{align*}
where $I_R^{(m)}(x)$ denotes the $m$-th derivative of $I_R(x)$.
\end{theorem}

\begin{remark}
	Another approach to establish the complete monotonicity of $I_R(x)$, given by Theorem \ref{ramanujan int cm thm}, is by using Theorem \ref{Bernstein theorem},  because the representation given in \eqref{ramanujan int eq} makes it evident that $I_R(x)$ is the Laplace transform of a non-negative measure.
\end{remark}

\begin{theorem}\label{ram int derivative cm thm}
For any non-negative integer $n$, the function $(-1)^n I_R^{(n)}(x)$ satisfies the following.
\begin{enumerate}[label=(\roman*)]
		\rm \item For all $x, y > 0$ and $\alpha \in (0,1)$, we have
\begin{align}\label{log conv rama int eq}
			(-1)^n I_R^{(n)}\left(\alpha x + (1 - \alpha) y\right) \leq \left((-1)^n I_R^{(n)}(x)\right)^\alpha \left((-1)^n I_R^{(n)}(y)\right)^{1 - \alpha}.
\end{align}
		
\item For $\alpha = \dfrac{1}{2}$ and $y=x+2$, we have
\begin{align}\label{turan rama int eq}
			\left(I_R^{(n)}(x+1)\right)^2 \leq I_R^{(n)}(x)  I_R^{(n)}(x+2),\quad x>0.
\end{align}
\end{enumerate}
\end{theorem}

\begin{proof}
	By Theorem \ref{ramanujan int cm thm}, $I_R(x)$ is completely monotone. Consequently, $(-1)^n I_R^{(n)}(x)$ is also completely monotone on $(0,\infty)$ for every nonnegative integer $n$. Since complete monotonicity implies logarithmic convexity \cite{Widder_1941_The Laplace Transform}, it follows that
\begin{align*}
		\log \left((-1)^n I_R^{(n)} \left(\alpha x + (1 - \alpha) y\right)\right) \leq \alpha \log \left( (-1)^n I_R^{(n)}(x)\right) + (1 - \alpha) \log \left( (-1)^n I_R^{(n)}(y)\right),
\end{align*}
for all $x, y > 0$ and $\alpha \in (0,1)$.
Simplifying the above inequality yields \eqref{log conv rama int eq}, thereby establishing the first part.
	
A straightforward substitution of $\alpha = \dfrac{1}{2}$ and $y = x + 2$ into \eqref{log conv rama int eq} yields \eqref{turan rama int eq}.
\end{proof}

\begin{remark}
	For $n = 0$, the inequalities \eqref{log conv rama int eq} and \eqref{turan rama int eq} reduce to the inequalities satisfied by the Ramanujan integral $I_R(x)$.
\end{remark}

The class of completely monotonic functions include several important subclasses, such as the class of Stieltjes functions, strongly completely monotonic functions, among others. In this work, we focus on a specific subclass of completely monotonic functions, known as the strongly completely monotone functions, which is defined as follows.

\begin{definition}\cite{pedresan_2009_scm_jmaa}
A function $f : (0, \infty) \to \mathbb{R}$ is said to be strongly completely monotonic function, if it is infinitely differentiable and for every integer $n \geq 0$,
$(-1)^n x^{n+1} f^{(n)}(x)$ is non-negative and strictly decreasing on the interval $(0, \infty)$.
\end{definition}

In \cite{trimble_1989_scm_siam}, an alternative characterization of strongly completely monotonic functions is provided, which is given below as a definition.

\begin{definition}\label{scm def}\cite{trimble_1989_scm_siam}
	A function $f$ is strongly completely monotonic if, and only if,
	\begin{align*}
		f(x) = \int_0^\infty e^{-x t} m(t) dt, \quad x\in(0,\infty),
	\end{align*}
	where $m(t)$ is a non-negative and non-decreasing function.
\end{definition}

Strongly completely monotonic functions exhibit several important properties, including subadditivity (superadditivity) and star-shapedness. In particular, the superadditive property plays an important role in statistical inference, as shown in \cite{trimble_1989_scm_siam}. More detailed studies on strongly completely monotonic functions can be found in  \cite{pedresan_2009_scm_jmaa}.

In the sequel, we examine whether the Ramanujan integral, known to be completely monotonic, also belongs to the subclass of strongly completely monotonic functions.


\begin{theorem}\label{rama int scm thm}
	For integers $n \geq 0$, consider the function
\begin{align}\label{rama int nth deri}
		(-1)^n I_R^{(n)}(x) = \int_{0}^{\infty} e^{-xt} \dfrac{t^{n-1}}{\pi^2 + \log^2 t} dt.
\end{align}
	Then, we have the following.
\begin{enumerate}[label=(\roman*)]
		\rm
\item For $n \geq 2$, $(-1)^n I_R^{(n)}(x)$ is strongly completely monotone on the interval $(0, \infty)$. However, $(-1)^n I_R^{(n)}(x)$ is not a Stieltjes function.
		
\item For $n=1$, $- I_R^{'}(x)$ is neither strongly completely monotone nor a Stieltjes function on $(0,\infty)$.
		
\item For $n=0$, $I_R(x)$ is not strongly completely monotone on $(0,\infty)$. However, precise conclusion for Stieltjes function behaviour of $I_R(x)$ is not available.
\end{enumerate}
	
\end{theorem}		
\begin{proof}
Let us assume
\begin{align*}
		\phi_n(t) = \dfrac{t^{n-1}}{\pi^2 + \log^2 t}, \quad t>0.
\end{align*}

Differentiating $\phi_n(t)$ with respect to $t$, we obtain
	
\begin{align}\label{value of phi deri}
		\phi_n'(t) = \dfrac{t^{n-2}}{\pi^2 + \log^2 t}  \underbrace{\left((n - 1)(\pi^2 + \log^2 t) - 2 \log t\right)}_{I_1}, \quad t>0.
\end{align}
	
The expression $I_1$ from \eqref{value of phi deri} can be rewritten as
\begin{align*}
		I_1 = \underbrace{(n - 1)\pi^2 - \dfrac{1}{n - 1}}_{I_2} + \left((n - 1)^{1/2} \log t - \dfrac{1}{(n - 1)^{1/2}}\right)^2, \quad t>0.
\end{align*}
	Since $I_2 > 0$ for all $n \geq 2$, it follows that $I_1 \geq 0$. This gives $\phi_n'(t) \geq 0$ for all $n \geq 2$, showing that $\phi_n(t)$ is non-decreasing on $(0,\infty)$. Moreover, $\phi_n(t)$ is nonnegative for all $n \geq 2$. Hence, by Theorem \ref{scm def}, $(-1)^n I_R^{(n)}(x)$ is strongly completely monotone for all integers $n \geq 2$. Further, the function $(-1)^n I_R^{(n)}(x)$ cannot be a Stieltjes function for any $n \geq 2$ because the density function $\phi_n(t)$ in its integral representation \eqref{rama int nth deri} is increasing for these values of $n$. However, for a function to be Stieltjes function, its density $\phi_n(t)$ must be decreasing \cite{Koumandos_survey on cm} which proves the theorem for $n\geq2$.
	
For $n = 1$, we have
\begin{align}\label{density of rama int deri}
I_R'(x) = - \int_{0}^{\infty} e^{-xt} \dfrac{1}{\pi^2 + \log^2 t} \, dt, \quad x > 0,
\quad
{\mbox{
with density
}}
\quad
\phi_1(t) = \dfrac{1}{\pi^2 + \log^2 t}, \quad t>0.
\end{align}

For $-I_R'(x)$ to be strongly completely monotone, the function $\phi_1(t)$ must be increasing \cite{Koumandos_survey on cm}. On the other hand, for $-I_R'(x)$ to be a Stieltjes function, $\phi_1(t)$ must be completely monotone \cite{Koumandos_survey on cm}.
	
Differentiating density function $\phi_1(t)$ given in \eqref{density of rama int deri} with respect to $t$ gives
\begin{align*}
		\phi_1'(t) = -\dfrac{2 \log t}{t(\pi^2 + \log^2 t)}, \qquad t > 0.
\end{align*}

This derivative is positive for $t \in (1,\infty)$ and negative for $t \in (0,1)$. Consequently, $-I_R'(x)$ is neither Stieljes nor strongly completely monotone and this proves the second part. The behavior of $\phi_1'(t)$ can also be observed in the following graph.
\begin{figure}[H]
        \includegraphics[scale=0.7]{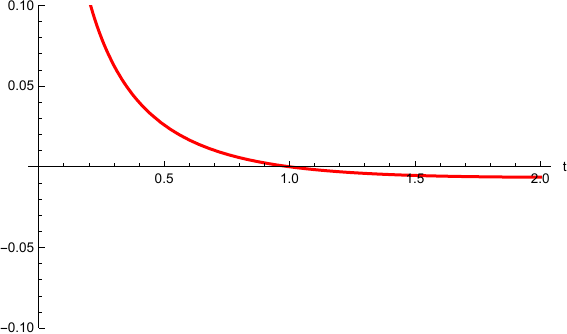}
		\caption{Graph of the derivative of $\phi_1(t)$}
		\label{fig: graph of derivative of phi_1(t)}
\end{figure}
	
For $n = 0$, the function $(-1)^n I_R^{(n)}(x)$ reduces to the Ramanujan integral $I_R(x)$ given in \eqref{ramanujan int eq}, with density
\begin{align*}
		\phi_0(t) = \dfrac{1}{t(\pi^2 + \log^2 t)}, \quad t > 0.
\end{align*}
For $I_R(x)$ to be a strongly completely monotone function, $\phi_0(t)$ must be increasing \cite{Koumandos_survey on cm}, whereas for $I_R(x)$ to be a Stieltjes function, $\phi_0(t)$ must be completely monotone \cite{Koumandos_survey on cm}.

Differentiating $\phi_0(t)$ with respect to $t$, we obtain
\begin{align*}
		\phi_0'(t) = -\dfrac{(\pi^2 - 1) + (\log t + 1)^2}{t^2 (\pi^2 + \log^2 t)^2}, \quad t>0,
\end{align*}
which is negative for all $t > 0$. This implies that $I_R(x)$ is not a strongly completely monotone function. However, whether it is a Stieltjes function remains to be determined, and we analyze this by observing the graphs of $\phi_0(t)$ and its derivatives presented below.

\begin{figure}[H]
        \includegraphics[scale=0.7]{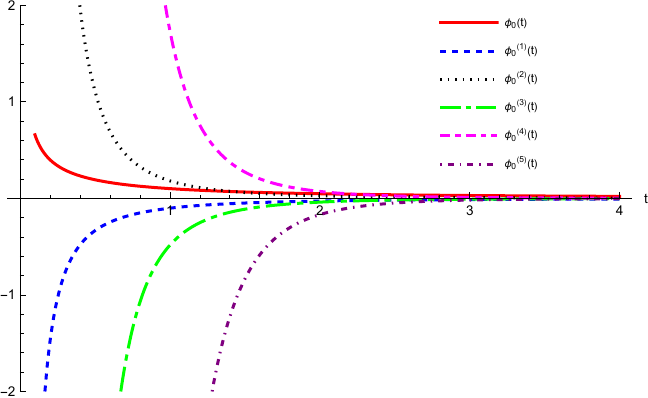}
		\caption{Graph of $\phi_0(t)$ upto its fifth derivatives}
		\label{fig: density of rama int}
\end{figure}
	
Figure \ref{fig: density of rama int} shows that $\phi_0(t)$ is positive and its first five derivatives alternate in sign. This suggests that $\phi_0(t)$ may be completely monotone, but no conclusion can be drawn without verifying this alternating sign property for all derivatives.
\end{proof}

A mathematical justification for the third part of Theorem \ref{rama int scm thm} is still needed. However it would be interesting to see whether $I_R^{(n)}(x)$ is a Stieltjes function. Once it is shown to be a Stieltjes function, we would have many properties to discuss further, such as whether it belongs to the class of Pick functions and related topics. Now we provide a theorem from \cite{trimble_1989_scm_siam} highlighting the implications of strong complete monotonicity.

%
%

\begin{lemma}\label{scm concequence}\cite{trimble_1989_scm_siam}
Suppose that $f$ is a strongly completely monotonic function on $(0,\infty)$ such that the product $xf(x)$ is not constant. Define a function $g$ by
\begin{align*}
		g(x) =
		\begin{cases}
			0, & \text{if } \  x = 0, \\
			\dfrac{1}{f(x)}, & \text{if } \  0 < x < \infty.
		\end{cases}
\end{align*}
Then the function $g$ is star-shaped, and consequently, it is superadditive.
\end{lemma}
Note that a continuous, non-negative function $f$ defined on $[0, \infty)$ with $f(0) = 0$ satisfying the property
\begin{align*}
	f(\beta x) = \beta f(x), \quad \forall \ \beta \in (0,1).
\end{align*}
is known as star-shaped, while if $f$ satisfies
\begin{align*}
	f(x + y) \geq f(x) + f(y), \quad \forall \ x, y \in (0, \infty),
\end{align*}
then, it is known to be superadditive.

Next, we discuss some bounds and related inequalities for $I_R^{(n)}(x)$.

\begin{theorem}
	Let $I_R^{(n)}(x)$ denote the $n$th derivative of the Ramanujan integral defined in \eqref{rama int nth deri}. Then the following inequalities hold.
\begin{enumerate}[label=(\roman*)]
		\rm 	\item \label{supadd for even n}
		For $n = 2m$, where $m = 1, 2, 3, \dots$, we have
\begin{align*}
			I_R^{(n)}(x) + I_R^{(n)}(y) \leq \dfrac{I_R^{(n)}(x) I_R^{(n)}(y)}{I_R^{(n)}(x + y)}, \quad x, y > 0.
\end{align*}
		
\item \label{supadd for odd n}
		For $n = 2m + 1$, where $m = 1, 2, 3, \dots$, we have
\begin{align*}
			I_R^{(n)}(x) + I_R^{(n)}(y) \geq \dfrac{I_R^{(n)}(x) I_R^{(n)}(y)}{I_R^{(n)}(x + y)}, \quad x, y > 0.
\end{align*}
		
\item \label{starshaped even n}
		For $n = 2m$, where $m = 1, 2, 3, \dots$, we have
\begin{align*}
			\beta I_R^{(n)}(\beta x) \leq I_R^{(n)}(x), \quad x > 0, \quad \beta \in (0, 1).
\end{align*}
		
\item \label{starshaped odd n}
		For $n = 2m + 1$, where $m = 1, 2, 3, \dots$, we have
\begin{align*}
			\beta I_R^{(n)}(\beta x) \geq I_R^{(n)}(x), \quad x > 0, \quad \beta \in (0, 1).
\end{align*}
\end{enumerate}
\end{theorem}

\begin{proof}
	Define the function $h(x)$ by
\begin{align*}
		h(x) =
		\begin{cases}
			0, & \text{if } x = 0, \\[6pt]
			\dfrac{1}{(-1)^n I_R^{(n)}(x)}, & \text{if } x > 0.
		\end{cases}
\end{align*}
Since $(-1)^n I_R^{(n)}(x)$ is strongly completely monotonic for all integers $n \geq 2$, it follows from Lemma \ref{scm concequence} that $h(x)$ is star-shaped. Consequently, $h(x)$ satisfy the inequalities
\begin{align}\label{superadditive rama int}
		\dfrac{1}{(-1)^n I_R^{(n)}(x)} + \dfrac{1}{(-1)^n I_R^{(n)}(y)} \leq \dfrac{1}{(-1)^n I_R^{(n)}(x+y)},
\end{align}
	and
\begin{align}\label{star shaped rama int}
		\dfrac{1}{(-1)^n I_R^{(n)}(\beta x)} \leq \beta \, \dfrac{1}{(-1)^n I_R^{(n)}(x)}, \quad \beta \in (0,1).
\end{align}
	Writing $n = 2m$ in \eqref{superadditive rama int} gives the first part and $n = 2m + 1$ reduces \eqref{superadditive rama int} to the second part. In similar lines, substitutions $n = 2m$ and $n = 2m + 1$ in \eqref{star shaped rama int} establish respectively the third and fourth part of the result.
\end{proof}

So far, we have discussed the class of completely monotone functions and its subclass, the strongly completely monotone functions. We now turn to another class, closely related to completely monotone functions, known as the class of Bernstein functions.

\begin{definition}\cite{Widder_1941_The Laplace Transform}
	A function $f : (0, \infty) \to \mathbb{R}$ is called a \emph{Bernstein function} if it is non-negative, infinitely differentiable, and its derivatives satisfy
\begin{align*}
		(-1)^n f^{(n)}(x) \geq 0 \quad \text{for all } n \in \mathbb{N}.
\end{align*}
\end{definition}
In other words, a function is called a Bernstein function if its derivative is completely monotonic. However, the converse does not always hold, that is, if a function is completely monotonic, its antiderivative may or may not be a Bernstein function. For example, the completely monotonic function $\dfrac{1}{x^2}$ on $(0,\infty)$ has the primitive $-\dfrac{1}{x}$, which is not a Bernstein function on $(0,\infty)$ because it is not non-negative.

In the following, we present a theorem that provides a characterization of this converse.

\begin{theorem}\cite{Schilling_2012_bernstein function}\label{antideri bernstein thm}
	Let $f$ be a completely monotonic function given by
\begin{align*}
		f(x) = a + \int_0^\infty e^{-x t}  d\mu(t), \quad \text{with } a \geq 0 \text{ and } x \in (0, \infty).
\end{align*}
	Then $f$ admits an antiderivative $g$ that is a Bernstein function if, and only if, the representing measure $\mu$ satisfies the integrability condition
\begin{align*}
		\int_0^\infty \dfrac{1}{1 + t}  d\mu(t) < \infty.
\end{align*}
\end{theorem}

We now consider the antiderivative of the Ramanujan integral and show that it turns out to be a Bernstein function.

\begin{theorem}\label{rama int anti der berns}
	The Ramanujan integral $I_R(x)$, defined in \eqref{ramanujan int eq}, admits an antiderivative $\tilde{I}_R(x)$, which is a Bernstein function. Moreover, it possesses the following explicit representation
\begin{align*}
		\tilde{I}_R(x) = a + \int_{0}^{\infty} \left(1 - e^{-xt}\right) \dfrac{dt}{t\left(\pi^2 + \log^2 t\right)} \, ,
\end{align*}
	where $a$ is a non-negative constant given by $a = \tilde{I}_R(0+)$.
\end{theorem}

\begin{proof}
	Consider the integral
\begin{align}\label{int for anti deri barnes}
		I = \int_{0}^{\infty} \dfrac{dt}{t(1 + t)\left(\pi^2 + \log^2 t\right)} \, .
\end{align}
	First, we examine the convergence of the integral $I$ by substituting $\log t = u$ in \eqref{int for anti deri barnes}. Thus, \eqref{int for anti deri barnes} becomes
	
\begin{align*}
		I = \int_{-\infty}^{\infty} \dfrac{du}{(1 + e^u)(\pi^2 + u^2)}.
\end{align*}
	Splitting the integral at zero and rewriting gives
\begin{align*}
		I = \int_{0}^{\infty} \left( \dfrac{1}{1 + e^{-u}} + \dfrac{1}{1 + e^u} \right) \dfrac{du}{\pi^2 + u^2}= \int_{0}^{\infty} \dfrac{du}{\pi^2 + u^2},
\end{align*}
	which is equal to $\tfrac{1}{2}$, and this establishes the convergence of the integral given in \eqref{int for anti deri barnes}. Therefore, by Theorem \ref{antideri bernstein thm},  $I_R(x)$ admits an antiderivative which is a Bernstein function.

To determine the explicit expression for $\tilde{I}_R(x)$, we integrate \eqref{ramanujan int eq} and get
\begin{align*}
		\int_{0}^{x} I_R(y)  dy &= \int_{0}^{x} \left( \int_{0}^{\infty} e^{-yt} \dfrac{dt}{t\left( \pi^2 + \log^2 t \right)} \right) dy.
\end{align*}
	Applying Fubini's theorem along with some computations lead to the conclusion, which completes the proof.
\end{proof}

\begin{remark}
In Theorem \ref{rama int anti der berns}, we have established that the antiderivative of the Ramanujan integral, $\tilde{I}_R(x)$, is a Bernstein function; however, it remains open to determine whether it is also a complete Bernstein function.
\end{remark}
%


The next theorem concerns a Turan-type function defined in terms of the derivatives of the Ramanujan integral and establishes its complete monotonicity under certain parametric conditions.

Exploring Turan type inequalities for various classes of functions is one of the well known research. Initially investigations were carried out for finding the conditions under which the expression
\begin{align}\label{Turan-type-expression}
P_n^2(x)-P_{n+1}(x)P_{n-1}(x)
\end{align}
is positive for certain specific polynomials. Later it was extended to various classes of functions by replacing $P_n(x)$ to $f(x)$ in \eqref{Turan-type-expression} by taking appropriate modifications to study if this expression is positive or negative. Recent interest is on analyzing various other properties of \eqref{Turan-type-expression}. For example, in \cite{mehrez_ramanujan type entire fun}, this expression was studied extensively for specific Ramanujan type entire functions. Motivated by this, we are interested in finding the completely monotonic behaviour of the derivatives of the Ramanujan integral $I_R(x)$ satisfying the expression analogous to \eqref{Turan-type-expression} by  replacing $P_n(x)$ with our integral $I_R^{(n)}(x)$.
The proof of our next theorem on complete monotonicity of Turan type functions, requires the following result.

\begin{lemma}\label{g(x) decreasing lemma}
	For $x>0$, the function
	\begin{align*}
		g(x)=\dfrac{1}{x}-\dfrac{2 \log(x)}{x\left(\pi^2+\log ^2x\right)}
	\end{align*}
	is strictly decreasing.
\end{lemma}

\begin{proof}
	Differentiating $g(x)$ with respect to $x$ gives
	\begin{align*}
		g'(x)= -\dfrac{1}{x^2} - \dfrac{2 \left(\pi^2+\log^2 x\right)-2\log(x)\left(\pi^2+\log^2 x+2\log(x)\right)}{x^2 \left(\pi^2+\log^2 x\right)^2}.
	\end{align*}
	After simplification, this can be written as
	\begin{align*}
		g'(x)=-\dfrac{1}{x^2 \left(\pi^2+\log^2 x\right)^2}
		\underbrace{\left(\log^2 x(2\pi^2-4)+2\pi^2+(\log x -\log^2 x)^2+(\log x -\pi^2)^2\right)}_{p(x)}.
	\end{align*}
	Since $p(x)>0$ for all $x>0$, it follows that $g'(x)<0$. Hence, $g(x)$ is strictly decreasing on $(0,\infty)$.
\end{proof}
The decreasing behavior of the function can also be observed from the graph shown below.
\begin{figure}[H]
    \includegraphics[scale=0.7]{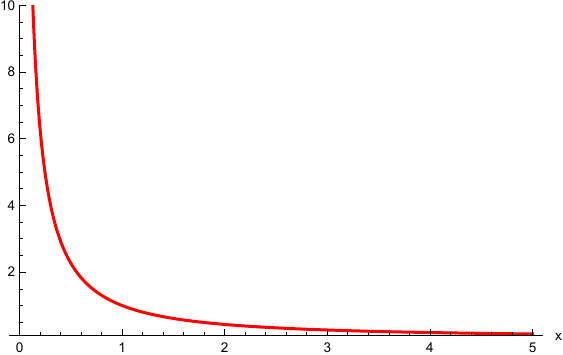}
	\caption{Graph of $g(x)$}
	\label{fig:decreasing nature of g(x)}
\end{figure}

\begin{theorem}\label{main theorem}
	Let $I_R^{(n)}(x)$ denote the $n$-th derivative of the Ramanujan integral defined in \eqref{rama int nth deri}.
For any integer $n \geq 2$ and $x>0$, define
\begin{align*}
	H_n(x;\alpha) = \left(I_R^{(n)}(x)\right)^2 - \alpha I_R^{(n-1)}(x) I_R^{(n+1)}(x), \qquad x>0.
\end{align*}
Then the following statements hold.
\begin{enumerate}[label=(\roman*)]
\rm \item \label{part 1}  If $H_n(x;\alpha)$ is completely monotonic on $(0,\infty)$, then the parameter $\alpha$ must satisfy
$
\displaystyle \alpha \leq \dfrac{n-1}{n}
$, for $n\geq 2$.
\item \label{part 2} If the parameter $\alpha$ is chosen such that
$
\displaystyle
\alpha \leq \dfrac{n-2}{n-1}
$,
$n\geq 2$,
then $H_n(x;\alpha)$ is completely monotonic on $(0,\infty)$.
\end{enumerate}
\end{theorem}

\begin{proof}
For the first part, by hypothesis, $H_n(x;\alpha)$ is completely monotonic. This, by Definition \ref{cmf definition}, means
	\begin{align*}
		H_n(x;\alpha) >0  \iff
		\alpha < \dfrac{ \left(I_R^{(n)}(x)\right)^2 }{ I_R^{(n-1)}(x)  I_R^{(n+1)}(x) }.
	\end{align*}
	This inequality can be rewritten in the form
\begin{align*}
		\alpha < \dfrac{ \left( x^n \log x  I_R^{(n)}(x) \right)^2 }{ \left( x^{n-1} \log x  I_R^{(n-1)}(x) \right) \left( x^{n+1} \log x  I_R^{(n+1)}(x) \right) }.
\end{align*}
	Taking the limit as $x \to \infty$, and using the asymptotic expansion given in \eqref{rama asym}, we obtain an equivalent inequality
\begin{align*}
		\alpha \leq \lim_{x \to \infty} \dfrac{ \left( \displaystyle\sum_{k=0}^{\infty} \dfrac{\Phi_k(\pi, n)}{\log^k x} \right)^2 }{ \left( \displaystyle\sum_{k=0}^{\infty} \dfrac{\Phi_k(\pi, n-1)}{\log^k x} \right) \left( \displaystyle\sum_{k=0}^{\infty} \dfrac{\Phi_k(\pi, n+1)}{\log ^k x} \right) }.
\end{align*}
Simplifying the above inequality, we get
\begin{align*}
		\alpha \leq  \lim_{x \to \infty} \dfrac{ \left( \dfrac{1}{\log^2 x} \displaystyle\sum_{k=1}^{\infty} \dfrac{\Phi_k(\pi, n)}{\log^{k-1} x} \right)^2 }{ \left( \dfrac{1}{\log^2 x} \displaystyle\sum_{k=1}^{\infty} \dfrac{\Phi_k(\pi, n-1)}{\log^{k-1} x} \right) \left( \dfrac{1}{\log^2 x} \displaystyle\sum_{k=1}^{\infty} \dfrac{\Phi_k(\pi, n+1)}{\log^{k-1} x} \right)},
\end{align*}
	which gives
\begin{align*}
		\alpha \leq  & \dfrac{ \left( \Phi_1(\pi, n) \right)^2 }{ \Phi_1(\pi, n-1) \Phi_1(\pi, n+1) }.
\end{align*}
	Substituting the values of $\Phi_1(\pi, n)$, $\Phi_1(\pi, n-1)$ and $\Phi_1(\pi, n+1)$ from \eqref{value of phi}, we get
	\begin{align*}
		\alpha \leq \dfrac{ \Gamma(n)^2 }{ \Gamma(n-1) \Gamma(n+1) }.
	\end{align*}
Solving the right-hand side gives the desired upper bound on $\alpha$, which completes the proof of the first part.

To establish the second part, we first consider the case $\alpha = \tfrac{n-2}{n-1}$.
Clearly
\begin{align*}
	H_n\left(x;\dfrac{n-2}{n-1}\right) = \left(I_R^{(n)}(x)\right)^2 - \left(\dfrac{n-2}{n-1}\right) I_R^{(n-1)}(x) I_R^{(n+1)}(x), \quad x > 0.
\end{align*}

Using the integral representation of $I_R^{(n)}(x)$ given in \eqref{rama int nth deri}, the above expression becomes
\newline
$
\displaystyle
H_n\left(x;\dfrac{n-2}{n-1}\right)
$
\begin{align*}
	=
	\left( \int_{0}^{\infty} e^{-xt} \dfrac{t^{n-1}}{\pi^2 + \log^2 t} dt \right)^2 - \!\! \left(\dfrac{n-2}{n-1}\right)
	\!\!\left( \int_{0}^{\infty} e^{-xt} \dfrac{t^{n-2}}{\pi^2 + \log^2 t} dt \right)
	\!\!\left( \int_{0}^{\infty} e^{-xt} \dfrac{t^{n}}{\pi^2 + \log^2 t} dt \right).
\end{align*}
Applying the convolution property of the Laplace transform, we obtain
\begin{align*}
H_n\left(x;\dfrac{n-2}{n-1}\right)= \int_{0}^{\infty} e^{-xt} h_n(t) dt	
\end{align*}
where
\begin{align}\label{value of h}
h_n(t)= \dfrac{1}{(n-1)}\int_{0}^{t}\left(\dfrac{u^{n-2}}{\pi^2 +\log^2 u}\right) \left(\dfrac{(t-u)^{n-1}}{\pi^2 +\log^2 (t-u)}\right)  \left((2n-3)u-(n-2)t \right)du
\end{align}
Now, in order to establish the complete monotonicity of
$H_n\left(x;\tfrac{n-2}{n-1}\right)$, we only need to verify that $h_n(t) \geq 0$ for all $t>0$.  For that, we substitute $u = \tfrac{t}{2}(1+v)$ in \eqref{value of h}, so that the integral becomes
\begin{align*}
h_n(t)= \dfrac{t^3}{8(n-1)}\int_{-1}^{1} \psi_n(v) dv,
\end{align*}
where
\newline
$
\displaystyle
\psi_n(v)
$
\begin{align*}
= \left(\dfrac{\left(\dfrac{t}{2}(1+v)\right)^{n-2}}{\pi^2 +\log^2 \left(\dfrac{t}{2}(1+v)\right)}\right) \left(\dfrac{\left(\dfrac{t}{2}(1-v)\right)^{n-2}}{\pi^2 +\log^2 \left(\dfrac{t}{2}(1-v)\right)}\right) \left(1-v^2(2n-3)  +2v(n-2) \right).
\end{align*}
Due to the symmetry of the integrand, the odd part vanishes upon integration, and the integral reduces to the even part. This gives
\newline
$
\displaystyle
h_n(t)
$
\begin{align*}
= \dfrac{t^3}{4(n-1)}\int_{0}^{1} \left(\dfrac{\left(\dfrac{t}{2}(1+v)\right)^{n-2}}{\pi^2 +\log^2 \left(\dfrac{t}{2}(1+v)\right)}\right) \left(\dfrac{\left(\dfrac{t}{2}(1-v)\right)^{n-2}}{\pi^2 +\log^2 \left(\dfrac{t}{2}(1-v)\right)}\right)  \left(1-(2n-3) v^2  \right) dv.
\end{align*}
We write the integral $h_n(t)$ as
\begin{align*}
	h_n(t)= \dfrac{t^3}{4(n-1)} \left(\dfrac{t}{2}\right)^{2(n-3)} \int_{0}^{1} I_1(v;t) I_2(v;t) dv,
\end{align*}
with $I_1(v;t)$ and $I_2(v;t)$ given respectively by
\begin{align}\label{value of I_1}
I_1(v;t)=	\left(\dfrac{\dfrac{t}{2}(1+v)}{\pi^2 +\log^2 \dfrac{t}{2}(1+v)}\right) \left(\dfrac{\dfrac{t}{2}(1-v)}{\pi^2 +\log^2 (\dfrac{t}{2}(1-v))}\right)
\end{align}
and
\begin{align} \label{value of I_2}
I_2(v;t)= (1-v^2)^{n-3} \left(1-(2n-3) v^2  \right) .
\end{align}
By taking logarithmic derivative of \eqref{value of I_1} with respect to $v$, we obtain
\begin{align*}
	\dfrac{I_1^{'}(v;t)}{I_1(v;t)}
	&= \dfrac{t}{2} \left( \dfrac{1}{\tfrac{t}{2}(1+v)}
	- \dfrac{2 \log\left(\tfrac{t}{2}(1+v)\right)}{\tfrac{t}{2}(1+v)\left(\pi^2+\log^2\left(\tfrac{t}{2}(1+v)\right)\right)} \right) \\
	&\quad - \dfrac{t}{2} \left( \dfrac{1}{\tfrac{t}{2}(1-v)}
	+ \dfrac{2 \log\left(\tfrac{t}{2}(1-v)\right)}{\tfrac{t}{2}(1-v)\left(\pi^2+\log^2\left(\tfrac{t}{2}(1-v)\right)\right)} \right).
\end{align*}
Since the function
 \begin{align*}
	\dfrac{1}{x}-\dfrac{2 \log(x)}{x\left(\pi^2+\log^2 x \right)}
\end{align*}
is decreasing, by Lemma \ref{g(x) decreasing lemma}, it follows that $I_1(v;t)$ is also decreasing. Hence, we can write
\begin{equation}\label{sign of I_1 in (0,1)}
	\begin{aligned}
		I_1(v;t) &\geq I_1\big((2n-3)^{-1/2};t\big),
		&& \text{for } 0 < v \leq (2n-3)^{-1/2}, \quad \text{and} \\[6pt]
		I_1(v;t) &\leq I_1\big((2n-3)^{-1/2};t\big),
		&& \text{for } (2n-3)^{-1/2} \leq v < 1.
	\end{aligned}
\end{equation}
For the function $I_2(v;t)$ defined in \eqref{value of I_2}, we have
\begin{equation}\label{sign of I_2 in (0,1)}
\begin{aligned}
		I_2(v;t) &> 0,  && \text{for } 0 < v < (2n-3)^{-1/2}, \quad \text{and} \\[6pt]
		I_2(v;t) &< 0,  && \text{for } (2n-3)^{-1/2} < v < 1.
\end{aligned}
\end{equation}
Thus, using \eqref{sign of I_1 in (0,1)} and \eqref{sign of I_2 in (0,1)}, we obtain
\begin{align*}
	I_1(v;t) I_2(v;t) \geq I_1\left((2n-3)^{-1/2};t\right) I_2(v;t),
	\quad v \neq (2n-3)^{-1/2},  v \in (0,1).
\end{align*}
Integrating over $(0,1)$ yields
\begin{align}\label{inequality involving I_1 and I_2}
	\int_{0}^{1} I_1(v;t) I_2(v;t) dv
	\geq I_1\left((2n-3)^{-1/2};t\right) \int_{0}^{1} I_2(v;t) dv.
\end{align}
Now, using \eqref{value of I_2}, we can write
\begin{align*}
	\int_0^1 I_2(v;t) dv
	&= \int_0^1 \dfrac{d}{dx}\left(-x(1-x^2)^{n-2}\right) dx=1.
\end{align*}
Thus, \eqref{inequality involving I_1 and I_2} simplifies to
$
\displaystyle	
\int_{0}^{1} I_1(v;t) I_2(v;t) dv \geq 0
$
.

Multiplying both sides by $\dfrac{t^3}{4(n-1)} \left(\dfrac{t}{2}\right)^{2(n-3)}$
gives the non-negativity of $h_n(t)$ for all $t>0$.
Hence, $H_n\left(x;\tfrac{n-2}{n-1}\right)$ is completely monotone on $(0,\infty)$.

Now, we can write, for $x>0$,
\newline
$
\displaystyle
H_n(x;\alpha)
$
\begin{align*}
	= \left(I_R^{(n)}(x)\right)^2
	+ \left(\dfrac{n-2}{n-1} - \alpha \right) I_R^{(n-1)}(x) I_R^{(n+1)}(x)
	- \left(\dfrac{n-2}{n-1}\right) I_R^{(n-1)}(x) I_R^{(n+1)}(x).
\end{align*}
Rearranging the terms gives, for $x>0$,
\begin{align*}
H_n(x;\alpha)
    & = \left(I_R^{(n)}(x)\right)^2 - \left(\dfrac{n-2}{n-1}\right) I_R^{(n-1)}(x) I_R^{(n+1)}(x)\\
	& + \left(\dfrac{n-2}{n-1} - \alpha \right)  (-1)^{n-1} I_R^{(n-1)}(x) (-1)^{n+1} I_R^{(n+1)}(x),
	\quad x>0.
\end{align*}
Using the complete monotonicity of $(-1)^{n} I_R^{(n)}(x)$ and of $H_n\left(x;\tfrac{n-2}{n-1}\right)$, together with the fact that any convex linear combination of completely monotone functions is itself completely monotone, we conclude that $H_n(x;\alpha)$ is also completely monotone.
\end{proof}
Since $H_n\left(x;\tfrac{n-2}{n-1}\right)$ is completely monotonic, we have
$	I_R^{(n)}(x))^2 - I_R^{(n-1)}(x) I_R^{(n+1)}(x) > 0,$
which provides the following specific case.
\begin{corollary}
	For $x>0$ and any natural number $n \geq 2$, we have
	\begin{align*}
		\dfrac{(I_R^{(n)}(x))^2}{I_R^{(n-1)}(x) I_R^{(n+1)}(x)} \geq \dfrac{n-2}{n-1}.
	\end{align*}
\end{corollary}

\begin{remark}
In Theorem \ref{main theorem}, we have established that $H_n(x;\alpha)$ is completely monotonic if, and only if, $\alpha \leq \dfrac{n-2}{n-1}$. However, the complete monotonicity of $H_n(x;\alpha)$ for $\alpha$ in the interval $\left(\dfrac{n-2}{n-1}, \dfrac{n-1}{n}\right)$ remains to be established. The graph below suggests that $H_n(x;\alpha)$ might be completely monotonic in this range.
\end{remark}

\begin{figure}[H]
    \includegraphics[scale=0.7]{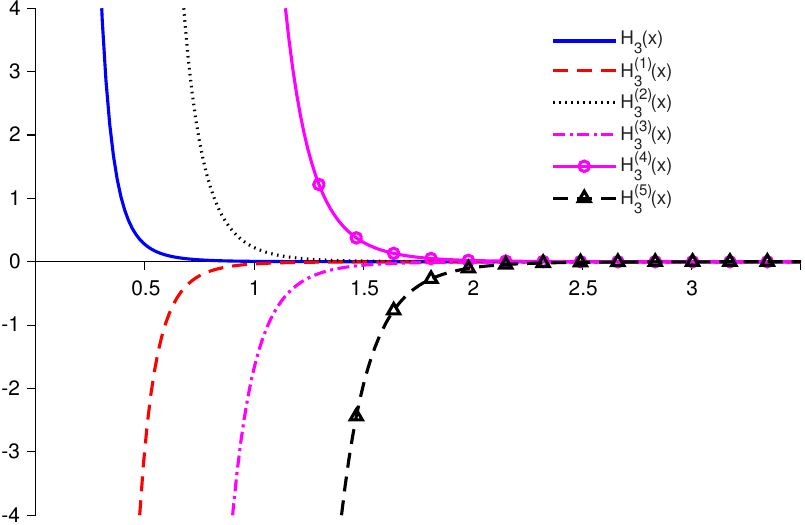}
	\caption{Graph of $H_3(x;0.55)$ upto its fifth derivative}
	\label{fig:graph of H_3(x)}
\end{figure}
From the Figure \ref{fig:graph of H_3(x)}, it can be observed that the function is positive and its derivatives alternate in sign, starting from negative sign. From the Definition \ref{cmf definition} of complete monotonicity, this suggests that $H_n(x;\alpha)$ may be completely monotonic within this interval. However, this observation is not conclusive, since the graph only presents derivatives up to the fifth order, whereas complete monotonicity requires verification for derivatives of all orders. Hence, we conclude with the following problem.

\begin{problem}
To establish the complete monotonicity of $H_n(x;\alpha)$ for $\alpha \in \left(\dfrac{n-2}{n-1}, \dfrac{n-1}{n}\right)$, where $n$ is a positive integer with  $n\geq2$.
\end{problem}

%
%
%
%
%
%
%
%
%

\end{document}